\documentclass[reqno]{amsart}
\usepackage{amssymb,amsmath,amsthm,amstext,amsfonts,color}
\usepackage[colorlinks=true,citecolor=black,linkcolor=black]{hyperref}
\usepackage[normalem]{ulem}
\newcommand{\uni}{UNI}
\newcommand{\hypersur}[1]{\mathcal{D}_{#1}}

\theoremstyle{plain}
\newtheorem{maintheorem}{Theorem}

\newtheorem{theorem}{Theorem } 

\newtheorem{lemma}[theorem]{Lemma}

\newtheorem{definition}[theorem]{Definition}
 \theoremstyle{remark}
\newtheorem{remark}[theorem]{Remark}

\newcounter{saveenum}

\newcommand{\field}[1]{\mathbb{#1}}

\newcommand{\RR}{\field{R}}

       \newcommand{\De}{\Delta}

\newcommand{\vep}{\varepsilon}
\renewcommand{\epsilon}{\varepsilon}

\newcommand{\la} {\lambda}

\newcommand{\vfi}{\varphi}

\newcommand{\cC}{\mathcal{C}}

\newcommand{\cJ}{\mathcal{J}}

\newcommand{\cR}{\mathcal{R}}

\newcommand{\cL}{\mathcal{L}}

\newcommand{\cF}{\mathcal{F}}

\newcommand{\cU}{\mathcal{U}}
\newcommand{\cS}{\mathcal{S}}

\def \fX {{\mathfrak X}}

\newcommand{\Leb}{\mbox{Leb}}
\newcommand{\meb}{\mathbf{m}}

\newcommand{\bR}{{\mathbb R}}
\newcommand{\bE}{{\mathbb E}}

\newcommand{\bN}{{\mathbb N}}

\newcommand{\norm}[1]{\left\| #1 \right\|}
\newcommand{\abs}[1]{\left| #1 \right|}

\newcommand{\M}{M}   			
\newcommand{\Sec}{\cS}    			
\newcommand{\U}{U}   			
\newcommand{\flow}[1]{X^{#1}} 	

\begin{document}

\title[Robust exponentialy mixing Axiom A flows]{Open sets
  of Axiom A flows with Exponentially Mixing Attractors}

\author{V. Ara\'ujo}
\address{V\'itor Ara\'ujo,
Departamento de Matem\'atica, 
  Universidade Federal da Bahia\\
  Av. Ademar de Barros s/n, 40170-110 Salvador, Brazil.}
\email{vitor.d.araujo@ufba.br \& vitor.araujo.im.ufba@gmail.com}

\author{O. Butterley}
\address{Oliver Butterley,  
Fakult\"at f\"ur Mathematik, Universit\"at Wien,
Oskar-Morgenstern-Platz 1, 1090 Wien, Austria.}
\email{oliver.butterley@univie.ac.at}

\author{P. Varandas} 
\address{Paulo Varandas, 
Departamento de Matem\'atica, 
  Universidade Federal da Bahia\\
  Av. Ademar de Barros s/n, 40170-110 Salvador, Brazil.}
\email{paulo.varandas@ufba.br}

\thanks{O.B. is grateful to Henk Bruin for several
  discussions.  O.B. acknowledges the support of the
  Austrian Science Fund, Lise Meitner position M1583.
  V.A. and P.V.  were partially supported by CNPq-Brazil,
  PRONEX-Dyn.Syst. and FAPESB (Brazil).  The authors are
  deeply grateful to Ian Melbourne for helpful advice and to
  the anonymous referees for their criticism and many
  suggestions that helped to improve the article. We are
  grateful to Fran\c{c}ois Ledrappier and Viviane Baladi for
  bringing some issues to our attention and for helpful
  further discussions that lead to the errata included here
  as an appendix.}

\date{\today}

\begin{abstract}
  For any dimension $d\geq 3$ we construct $\cC^{1}$-open
  subsets of the space of $\cC^{3}$ vector fields such that
  the flow associated to each vector field is Axiom~A and
  exhibits a non-trivial attractor which mixes exponentially
  with respect to the unique {SRB} measure.
\end{abstract}


\maketitle
\thispagestyle{empty}

\section{Introduction}
The Axiom~A flows of Smale~\cite{Sm67} have been extensively
studied in the last four decades and are now relatively well
understood.  One important remaining question concerns the
rate of mixing for such flows.  
Let $\M$ be a Riemannian manifold (dimension $d\geq 3$) and
 let $\Lambda$ be a basic set
for an Axiom A flow $X^t : \M \to \M$. This means that 
 $\Lambda \subset \M$ is an invariant, closed, topologically
transitive, locally maximal hyperbolic set. 
A basic set $\Lambda$ is called an
\emph{attractor} if there exists a neighbourhood  $\U$ of $\Lambda$, and $t_0>0$, such that
$\Lambda = \bigcap_{t\in  \bR_+}\flow{t} \U$.  A basic set is \emph{non-trivial} if
it is neither an equilibrium nor a periodic solution. 
In this article we focus on the {SRB} measure; this is the invariant probability measure which
is characterised by having absolutely continuous conditional measures 
on unstable manifolds. 
 It is known that for every attractor of an Axiom~A flow
there exists a unique {SRB} measure. 
In this setting it is known that the unique SRB measure
is also the unique physical measure and is also the 
Gibbs measure associated to the potential chosen as 
  minus the logarithm
of the unstable Jacobian~\cite{Young2002}.  
We say that an invariant probability measure $\mu$ is mixing if the
correlation function
 $\rho_{\phi,\psi}(t) := \int_{\Lambda} \phi\circ X^{t} \cdot \psi \
 d\mu - \int_{\Lambda} \phi \ d\mu \cdot \int_{\Lambda} \psi \ d\mu$
tends to zero for $t\to+\infty$, for all bounded measurable
observables $\phi,\psi: U \to\RR$.  We say that it mixes
exponentially if, for any fixed H\"older exponent $\alpha\in
(0,1)$, there exist $\gamma, C>0$ such that, for all
$\phi$, $\psi$ which are $\alpha$-H\"older on $U$, $\abs{
  \rho_{\phi,\psi}(t)} \leq C \norm{\phi}_{\cC^{\alpha}}
\norm{\psi}_{\cC^{\alpha}} e^{-\gamma t}$ for all $t\geq 0$.

The conjecture that all mixing Axiom~A flows,
with respect to Gibbs measures for
  H\"older continuous potentials, mix exponentially was
proven false by considering
suspension semiflows with piecewise constant return times Ruelle~\cite{ruelle1983}
and Pollicott~\cite{Pol85} constructed examples with arbitrarily slow mixing rates.
Building on work by Chernov~\cite{chernov98},
Dolgopyat~\cite{Do98} demonstrated that
$\cC^{2+\epsilon}$ transitive Anosov flows
  with $\cC^{1}$~stable and unstable foliations which are
  jointly nonintegrable mix exponentially with respect to
  the SRB measure and
H\"older continuous observables.
Unfortunately this good regularity of the invariant
foliations is not typical~\cite{HW1999}.  Dolgopyat also
showed~\cite{dolgopyat98} that rapid mixing (superpolynomial) is typical, in a measure theoretic sense of
prevalence, for Axiom~A flows with respect to any
equilibrium state associated to a H\"older potential.
Building on these ideas Field, Melbourne and
T\"or\"ok~\cite{FMT} showed that there exist $\cC^{2}$-open,
$\cC^{r}$-dense sets of $\cC^{r}$-Axiom~A flows ($r\geq 2$)
for which each non-trivial basic set is rapid mixing. 

 It is tempting to think that exponential mixing is a robust
property, i.e., if an Axiom~A flow mixes exponentially then
all sufficiently close Axiom~A flows also mix
exponentially. This remains an open problem, even limited to
the case of Anosov flows.
Dolgopyat~\cite{Dolgopyat2000} studied suspension semiflows over 
  topologically mixing subshifts of finite type with respect to Gibbs measures and
  showed that an open and dense subset (in the H\"older topology) mix exponentially. 
  He conjectured \cite[Conjecture~1]{Dolgopyat2000} that the set of exponentially mixing flows
  contains a $\cC^{r}$-open, $\cC^{r}$-dense subset of the set of
  all Axiom~A flows. 
  Unfortunately the previous mentioned result does not help in proving this conjecture
   since the H\"older topology is pathological for these purposes 
  (for details see the second remark after Theorem 1.1 in \cite{Dolgopyat2000}).
  Liverani~\cite{liverani2004} was able to
  bypass the regularity of the stable and unstable invariant
  foliations and show exponential mixing for $\cC^4$ contact
  Anosov flows using the contact structure, an observation
  that suggests that the regularity of the invariant foliations
  is not essential.  Unfortunately, contact flows are a thin
  subset of Anosov flows (or Axiom~A flows), in particular
  there do not exist open subsets of Axiom~A flows which preserve a
  contact structure.  
 Exponential rates of mixing were proved for $\cC^2$ uniformly expanding surface suspension semiflows by  Baladi and Vall\'ee~\cite{BaVal2005} under the assumption
 that the return time function is cohomologous to a piecewise constant function.
 This was extended to arbitrary dimension by Avila, Gou\"ezel, and Yoccoz~\cite{AvGoYoc}.
In a recent preprint, Ara\'ujo and Melbourne \cite{AM2015} 
extended~\cite{BaVal2005} relaxing $C^2$ to $C^{1+\alpha}$.
These results for non-invertible flows can be applied to invertible systems
 (e.g., \cite{AvGoYoc}) when the 
stable foliation is of sufficient regularity.
  For a more complete history of the
  question  of mixing rates of hyperbolic flows 
 we refer to the reader  the introductions
  of~\cite{liverani2004, FMT}.

The purpose of this paper is to construct open sets of
Axiom~A flows which mix exponentially with
  respect to the SRB measure of its attractor, thus
  making the step from open sets of symbolic flows to open
  sets of Axiom A flows at the cost of assuming that the
  stable foliation is $\cC^2$-smooth. The existence of such
  open sets of Axiom~A flows was expected
 \cite{Dolgopyat2000}.
   The first and third author
previously constructed~\cite{ArVar} open sets of
three-dimensional singular hyperbolic flows (geometric
Lorenz attractors) which mix exponentially with respect to
the unique SRB measure.  It was unclear if the singularity 
actually aids the mixing and allows for the robust
exponential mixing. In this article we show that we do not
need  the
singularity but actually we can just take advantage of the
volume contraction of the flow (and consequently a
domination condition) in order to carry out a similar
construction.

\section{Results \& Outline of the Proof}

Given an Axiom~A flow $\flow{t} : \M \to \M$
associated to a vector field $X$ we consider
the $\cC^{r}$ distance on the space of $\cC^r$-vector fields
$\fX^r(\M)$, that induces a natural distance on
the space of flows.  
The following two theorems are the main results of this article.

\begin{maintheorem}\label{thm:main}
  Given any  Riemannian manifold $\M$ of dimension
  $d\geq 3$ there exists a $\cC^{1}$-open subset of
  $\cC^3$-vector fields $\cU \subset \fX^3(M)$ such that for
  each $X\in\cU$ the associated flow is Axiom~A and exhibits
  a non-trivial attractor which mixes exponentially with
  respect to the unique SRB measure.
\end{maintheorem}
\noindent
As far as the authors are aware, this is the first
result concerning the existence of robustly exponentially
mixing Axiom~A flows.  The strategy for the construction of
the open sets in the above theorem is similar to the one
developed in \cite{ArVar} for singular flows.
Theorem~\ref{thm:main} is a consequence (details in
Section~\ref{sec:mainproof}) of the following more
fundamental result.

 \begin{maintheorem}\label{thm:expmix}
   Suppose that $\flow{t} : \M \to \M$ is a $\cC^2$ Axiom~A
   flow, $\Lambda$ is an attractor, and  that the stable foliation is
   $\cC^2$.  
   If the stable and unstable foliations
    are not jointly integrable
    then the flow mixes exponentially with
   respect to the unique SRB measure for
     $\Lambda$.
 \end{maintheorem}

\noindent
The proof of the above is described in
Section~\ref{sec:expmix} and involves quotienting along
stable manifolds of a well-chosen Poincar\'e section to
reduce to the case of a suspension semiflow over an
expanding Markov map. We can then apply the result of 
~\cite{AvGoYoc}
which implies
exponential mixing for the semiflow unless the return time
function is cohomologous to a piecewise constant function.
This is then related to the exponential mixing for the original flow
and the joint integrability of the stable and unstable foliations.
It is expected that~\cite{AM2015} would allow the $\cC^{2}$ requirement 
for the stable foliation to be weakened to $\cC^{1+\alpha}$.
We have no reason to believe that the requirement of good regularity 
(better than $\cC^1$) of the
 stable foliation is essential to the above
theorem, however the present methods rely heavily on this
fact. Note that although we require this good regularity of
the stable foliation we have no requirements on the
regularity of the unstable foliation. 
We observe that the required good regularity of the stable foliation can hold robustly, in contrast to
Dolgopyat's original argument~\cite{Do98} which required $\cC^{1}$
regularity for both the stable and unstable foliations.
 
The following questions remain: Are all exponentially mixing
Axiom~A flows also robustly exponentially mixing?  And the above mentioned 
conjecture, for any $r > 1$, does the set of exponentially mixing Axiom A flows
contain an $\cC^{r}$-open and dense subset of the set of all
Axiom A flows? 
It would appear that both these questions are of a higher order of
difficulty.
 
 The joint nonintegrability of stable and unstable foliations (as 
 required in Theorem~\ref{thm:expmix}) can be
seen in several different ways.  The stable and unstable
foliations of an Axiom~A flow are always transversal,
consequently, if they are jointly integrable, this provides
a codimension one invariant foliation which is transversal
to the flow direction. Conversely, if there exists a
codimension one invariant foliation which is transversal to
the flow direction, then this foliation must be subfoliated
by both the stable and unstable foliations which must
therefore be jointly integrable.  
In this case
it is known~\cite[Proposition~3.3]{FMT} that the flow is
(bounded-to-one) semiconjugate to a locally constant
suspension over a subshift of finite type.  
Such a flow need not mix, or may mix slower than exponentially.

The additional ingredient in order to use
Theorem~\ref{thm:expmix} to prove Theorem~\ref{thm:main} is
a result concerning the regularity of foliations. 
  Let $\|\cdot\|$
denote the Riemannian norm on the tangent space of $\M$.  As
before $\flow{t}:\M \to \M$ is an Axiom~A flow and $\Lambda$
is a non-trivial basic set. Since the flow is Axiom~A the
tangent bundle restricted to $\Lambda$ can be written as the
sum of three $DX^{t}$-invariant continuous subbundles,
$T_{\Lambda}\M = \bE^{s} \oplus \bE^{c} \oplus \bE^{u} $
where $ \bE^{c} $ is the one-dimensional bundle tangent to
the flow and there exists $C,\lambda >0$ such that $\norm{
  \left. D\flow{t}\right|_{\bE^{s}} } \leq Ce^{-\lambda t}$,
and $\norm{ \left. D\flow{-t}\right|_{\bE^{u}} } \leq
Ce^{-\lambda t}$, for all $t\geq 0$.
 
\begin{theorem}
 \label{thm:regularity}
 Suppose that $\Lambda$ is an attractor for the $\cC^3$
 Axiom~A flow $\flow{t} : \M \to \M$.  Further suppose that
  \begin{equation}\label{eq:kdomination}
\sup_{x\in \Lambda} \norm{ \left. D\flow{t_0}_{x}\right|_{\bE^{s}}  } \cdot 
 \norm{ \smash{ D\flow{t_0}_{x} }  }^{2} < 1,
 \end{equation}
 for some $t_0 >0$. 
 Then the stable foliation of $X^{t}$ is $\cC^{2}$.
 \end{theorem}
 
 \noindent
 This is a well-known consequence of the
 arguments
   described by Hirsch, Pugh, and
 Shub~\cite{HPS77} and also in the proof of
 \cite[Theorem~6.1]{HP70}.
We refer the reader to Section~6 in \cite{PSW} for a discussion
on three definitions for $C^k$-smoothnesss of foliations.
We observe that in our context, where $k>1$, the results from \cite{HP70,HPS77} provide the strongest of these
three notions, where the smoothness of the stable foliation is obtained
when it is regarded as a section into a certain Grassmannian (cf.\@ \cite{PSW}). 
This is enough to guarantee that local cross-sections obtained by collecting 
local stable manifolds through a local unstable manifold and stable holonomies
are indeed $C^2$ smooth (cf.\@ construction of Poincar\'e sections and return maps 
in Section~\ref{sec:expmix}).
 
This question of regularity is a subtle
issue. For Anosov flows, one cannot in general
expect the stable foliation (or equivalently the unstable
foliation) to have better regularity than
H\"older~\cite{HW1999}.  If the invariant
  foliation of interest has codimension one, then $\cC^1$
  regularity can be obtained; see e.g. \cite[Appendix
  1]{PT93}.  
  However the stable foliation of a hyperbolic
  flow can never have codimension one since in the splitting
  $T_{\Lambda}\M = \bE^{s} \oplus \bE^{c} \oplus \bE^{u}$
  the complementary direction to $\bE^s$ is
  $\bE^c\oplus\bE^u$.

We will construct open sets of Axiom~A flows which satisfy
the assumptions of Theorem~\ref{thm:regularity}.
Note that the domination condition \eqref{eq:kdomination}
     implies that the flows we consider are volume
     contracting and so the attractor has zero volume and
     consequently the flow is necessarily not Anosov.
  Since the domination
   condition 
   is an open condition it implies that there exist open
   sets of Axiom~A flows which possess a $\cC^{2}$ stable
   foliation.  
   Since the non-joint integrability of stable and unstable foliations
   is an open and dense condition (see e.g.~\cite[Remark 1.10]{FMT} and references therein), 
   to prove that robust exponential mixing does exist we apply Theorem~\ref{thm:expmix}. 
  Details of this argument are given in
   Section~\ref{sec:mainproof}.

 \section{Axiom~{A} attractors with $C^2$-stable foliation.}
 \label{sec:expmix}

 The purpose of this section is to prove
 Theorem~\ref{thm:expmix}. For the duration of this section
 we fix $X \in \fX^{2}(\M)$ where $M$ is $d$-dimensional and
 assume that $\flow{t} : \M \to \M$ is a $\cC^2$ Axiom~A
 flow and $\Lambda\subset M$ is an
 attractor (in particular  topologically transitive) with a
 $\cC^2$ stable foliation.  
 
 Given $x\in \Lambda$ we denote
 by
  $ W_{\epsilon}^{s}(x)=\{y\in M : d(X^t(y), X^t(x)) \le
 \epsilon, \; \forall t\ge 0 \text{ and } d(X^{t}(y), X^{t}(x)) \to 0 \text{ as } t\to +\infty \} $
  the \emph{local (strong)
   stable manifold} of $x$ which is forward invariant, that
 is, $X^t(W^s_\epsilon(x)) \subset W^s_\epsilon (X^t(x))$
 for every $x\in \Lambda$ and $t\ge 0$.  Consider also the
 \emph{local centre-stable manifold} of $x$ defined as $
 W_{\epsilon}^{cs}(x) =\cup_{|t|\le \vep}
 X^t(W^s_\epsilon(x)) $.
 The local (strong) unstable and
 centre-unstable manifolds, $W_{\epsilon}^{u}(x)$ and
 $W_{\epsilon}^{cu}(x)$ respectively, correspond to the
 local strong stable and local centre-stable manifolds for the flow
 $(X^{-t})_t$.  It is known~\cite{Bo73} that the
 attractor $\Lambda$ has a local product structure.  This
 means that there exists an open neighbourhood $\cJ$ of the
 diagonal of $\M \times \M$ and $\epsilon>0$ such that, for
 all $(x,y) \in \cJ_{\Lambda} := \cJ \cap (\Lambda \times
 \Lambda)$, then $W_{\epsilon}^{cu}(x) \cap
 W_{\epsilon}^{s}(y) \neq\emptyset$ and $W_{\epsilon}^{u}(x)
 \cap W_{\epsilon}^{cs}(y) \neq\emptyset$ each consist of a
 single point and this intersection point
 belongs to $\Lambda$.  
 In this case it makes sense to
 consider the continuous maps $\left[ \cdot,
   \cdot\right]_{s} : \cJ_{\Lambda} \to \Lambda$ and $\left[
   \cdot, \cdot\right]_{u} : \cJ_{\Lambda} \to \Lambda$
 defined by
\[
W_{\epsilon}^{cu}(x) \cap W_{\epsilon}^{s}(y) = \left\{ \left[ x, y\right]_{s} \right\},
\quad
W_{\epsilon}^{u}(x) \cap W_{\epsilon}^{cs}(y) = \left\{ \left[ x, y\right]_{u} \right\}.
\]
Following~\cite[\S 4.1]{FMT}
  we note that the set $\cJ$, and $\epsilon>0$ may be
  chosen fixed for some $\cC^{1}$-open set $\mathcal{U}
  \subset \fX^{1}(\M)$ containing any given initial Axiom A flow.
\begin{definition}[{\cite{Bo73}}]
A differentiable closed $(d-1)$-dimensional disk $\cS \subset \M$, transverse
to the flow direction is called a \emph{local
  cross-section}.  A set $\cR \subset \Lambda \cap \cS$ is called
a \emph{rectangle} if $W_{\epsilon}^{cs}(x) \cap
W_{\epsilon}^{cu}(y) \cap \cR $ consists of exactly one
point for all $x,y\in \cR$.  
\end{definition}
Let $\cR_{i}^{*}$ denote the
interior of $\cR_{i}$ as a subset of the metric space
$\cS_{i} \cap \Lambda$.  
\begin{definition}[{\cite{Bo73}}]
A finite set of rectangles
$\mathbf{R} = {\{ \cR_i \}}_i$ is called a \emph{proper
  family} (of size $\epsilon$) if
 \begin{enumerate}
 \item
 $\Lambda = \cup_{t\in [-\epsilon,0]} X^{t}  \left(\cup_{i}  \cR_{i}\right)=: X^{[-\epsilon,0]} \left(\cup_{i}  \cR_{i}\right)$,
  \setcounter{saveenum}{\value{enumi}}
\end{enumerate}
and there exist local sections ${\{\cS_{i} \}}_{i} $ of
diameter less than $\epsilon$ such that
\begin{enumerate}
  \setcounter{enumi}{\value{saveenum}}
  \item
  $\cR_{i} \subset \operatorname{int}(\cS_{i})$
  and $\cR_{i} = \overline{\cR_{i}^{*}}$,
 \item
 For $i\neq j$, at least one of the sets $\cS_{i}  \cap \flow{[0,\epsilon]} \cS_{j}$ and 
 $\cS_{j} \cap \flow{[0,\epsilon]} \cS_{i}$ is empty.
  \end{enumerate} 
  \end{definition} 
  Given a proper family as above, let $\Gamma := \cup_{i} \cR_{i}$, and denote by $P$ the Poincar\'e return
  map to $\Gamma$ associated to the flow $\flow{t}$, and by $\tau$ the return time. 
Although  $P$ and $\tau$ are not continuous on $\Gamma$, they are continuous on
  \[
  \Gamma' := \left\{  x\in \Gamma : P^{k}(x) \in \smash{ \cup_{i} } \cR_{i}^{*} \text{ for all $k\in \mathbb{Z}$}\right\}.
  \]
  \begin{definition}[{\cite{Bo73}}] 
    A proper family $\mathbf{R} = {\{ \cR_i \}}_i$ is a
    \emph{Markov family} if
\begin{enumerate}
\item
$x \in U(\cR_{i}, \cR_{j}) :=   \overline{\left\{   w \in \Gamma' : w\in \cR_{i}, P(w) \in \cR_{j}  \right\} } $ implies  
$\cS_{i} \cap W_{\epsilon}^{cs}(x) \subset U(\cR_{i},\cR_{j})$,
\item
$y \in V(\cR_{k}, \cR_{i}) :=   \overline{\left\{   w \in \Gamma' : P^{-1}(w)\in \cR_{k}, w  \in \cR_{i}  \right\} }$ implies  
$\cS_{i} \cap W_{\epsilon}^{cu}(y) \subset V(\cR_{k},\cR_{i})$.
\end{enumerate} 
\end{definition}
It follows from \cite[Theorem~2.4]{Bo73} that for any Markov
family $\mathbf R=\{\cR_i\}$ the flow is (bounded-to-one)  semiconjugate to a
suspension flow, with a bounded roof
  function also bounded away from zero, over the subshift
of finite type $\sigma_{\mathbf R}:\Sigma_{\mathbf R} \to
\Sigma_{\mathbf R}$ where $\Sigma_{\mathbf R} =\{ (a_i)_i\in
\mathbb Z : A_{a_i a_{j}}=1 \; \forall i, j \}$ and $A_{a_i
  a_{j}}=1$ if and only if there exists $x\in \Gamma' \cap
\cR_i$ and $P(x) \in \cR_j$.

The first step in the proof of Theorem~\ref{thm:expmix} is
to carefully choose local cross-sections for the
flow near the attracting basic set $\Lambda$.  Since
$\Lambda$ is an hyperbolic attractor then the local unstable
manifold of each point of the attractor is contained within
the attractor; see e.g.~\cite{Bo75}.  Hence for any
$x\in\Lambda$ and small enough $\epsilon>0$ we have
$W^u_\epsilon(x)\subset \Lambda$ and the $(d-1)$-submanifold
generated by the union of local stable manifolds through
points of $W^u_\epsilon(x)$,
\begin{align}\label{eq:localsec}
  \cS_\epsilon(x) = \bigcup_{y\in W^u_\epsilon(x)} W_\epsilon^{s}(y),
\end{align}
is a local cross-section containing $x$. 
Moreover, since the stable foliation of $\flow{t}:\Lambda \to \Lambda$ is
$\cC^2$ and $W^u_\epsilon(x)$ is a $\cC^2$-disk in $W^u(x)$,
then $\cS_\epsilon(x)$ is a codimension one
$\cC^2$-hypersurface foliated by $\cC^2$ local stable
manifolds. In addition, there exists a natural projection
$\pi_{\epsilon,x}:S_\epsilon(x)\to W^u_\epsilon(x)$ through
the stable leaves which is of class $C^2$, by construction.
In general there is no reason to expect
that these local cross-sections will be foliated by local unstable manifolds
even though it contains one local unstable manifold at the centre.

 \begin{lemma}\label{lem:markovfamily}
   Let $X^t: \M \to \M$ be an Axiom A flow 
     and $\Lambda\subset M$  an
     attractor.  There exists a finite number of
$\cC^2$ local cross-sections $\Sec = {\{
       \Sec_i\}}_i \subset \U$ such
     that the sets $\cR_i = \Sec_i \cap \Lambda$ are
     rectangles and $\mathbf{R} = {\{ \cR_i \}}_i$ is a
     Markov family for $\flow{t}$. Moreover, each
     hypersurface $\Sec_i$ is subfoliated by strong stable
     leaves and for each rectangle $\cR_i$ there exists $x_i
     \in \cR_i$ and a $\cC^2$-disk $\Delta_i\subset
     W^u_\epsilon(x_i)\subset \Lambda$ such that
     $\Sec_i=\cup_{y\in \Delta_i} \gamma^s(y)$ where
     $\gamma^s(y)$ is an open subset of $W_\epsilon^s(y)$
     that contains $y$. In addition, the projection
     $\pi_i:\Sec_i\to\Delta_i$ along stable manifolds is
     $C^2$ smooth.
 \end{lemma}

 \begin{proof}
   We have shown~\eqref{eq:localsec} that through each point
     $x$ of $\Lambda$ there passes a $C^2$ codimension one
     disk transversal to the flow, whose diameter can be
     made arbitrarily small, formed by the union of stable
     leaves through the points of $W^u_\epsilon(x)$, and
     these disks have a $C^2$ smooth projection along the
     stable leaves onto the unstable disk $W^u_\epsilon(x)$.
   This is the starting point of the proof
     of \cite[Theorem~2.5]{Bo73}, detailed in \cite[Section
     7]{Bo73}, to show that the $\flow{t}:\Lambda \to
     \Lambda$ admits a Markov family as in the statement of
     lemma, consisting of a finite number of rectangles of
     arbitrarily small size $\epsilon$ contained in the
     interior of cross-sectional disks $\Sec = {\{
       \Sec_i\}}_i \subset \U$, each one endowed with an
     unstable disk $\Delta_i\subset\Sec_i$ and a projection
     $\pi_i$, as in the statement.  

     The extra properties of these disks are consequences of
     the initial construction of smooth local cross-sections to
     the flow taking advantage of the fact that $\Lambda$ is
     an attractor with a $C^2$ smooth stable
     foliation.\footnote{As pointed out to us by
       Mark Pollicott, a similar idea to the above was used by
       Ruelle~\cite{Ru76a}, modifying the construction of
       Bowen~\cite{Bo73} to produce local sections with improved
       regularity.}
 \end{proof}
 
Let $\Sec$ be as given by Lemma~\ref{lem:markovfamily}.
   Now we consider the flow on $U \supset \Lambda$ as a suspension
   flow, return map $P: \Sec \to\Sec$ and 
     return time $\tau : \Sec \to [\underline\tau,\bar\tau]$
     for some fixed $0<\underline\tau<\bar\tau<\infty$.
   Let $\cF_{s}$ denote the foliation of $\U$ by local
   stable manifolds. 
   Let $\Delta = \Sec \diagup \cF_{s}$ (the
   quotient of $\Sec$ with respect to the local stable
   manifolds).  A concrete realization of this quotient is
   given by $\Delta = \cup_{i} { \Delta_i }$. 
   A key point in this construction is
   that the return time $\tau$ is constant along the local
   stable manifolds and $C^2$-smooth. Quotienting along
   these manifolds we obtain a suspension semiflow over an
   expanding map $f: \Delta \to \Delta$. The fact that, by
   Lemma~\ref{lem:markovfamily}, the local cross-section $ \Sec =
   {\{\Sec_i\}}_i$ is foliated by local stable manifolds is
   essential.  
    We write
   $\widetilde\meb$ for the normalised  restriction of the Riemannian volume 
    to the family $\Delta$ of $C^2$ disks ${\{\Delta_i\}}_i$.
   Since $\mathbf{R}$ is a Markov family, this ensures the
   Markov structure of $f: \Delta \to \Delta$ and hence, for
   each $\Delta_i$, there exists a partition
   $\{\Delta_{i,j}\}_{j}$ of a full $\widetilde\meb$-measure subset of
   $\Delta_{i}$ such that $f : {\Delta_{i,j}}\to\Delta_j$ is
   a bijection.  The $\cC^2$-smooth regularity of the local
   cross-sections $\Sec_{i}$ and the flow $X^t$ is enough to
   guarantee that the return map $P: \Sec \to \Sec$ is also
   $\cC^{2}$ on each component. Because the projection $\pi$
   is also $\cC^2$-smooth, we conclude that
   $f:\De_{i,j}\to\De_j$ is a $\cC^{2}$ diffeomorphism for
   each $i$, $j$.
For future
 convenience, let $\pi : \Sec \to \De$ denote the collection
 of the projections $\pi_i$, so that $f\circ \pi=\pi\circ
 P$. 
 Since the return time function is constant along stable
 leaves, we also denote the return time function on $\Delta$ by
 $\tau : \Delta \to \bR_{+}$.

 Subsequently we wish to make the connection to the flows
 studied in~\cite{AvGoYoc}. It is therefore convenient to
 work with a full branch expanding map whereas the quotient
 return map $f : \cup_{i,j} \Delta_{i,j} \to \Delta$ is a
 transitive Markov expanding map but might not be full
 branch.  We consider an induced map to guarantee the full
 branch property.  Let $F$ denote the first return map to
 some element $\De_{0}$ of the Markov partition of $f$ (the
 choice of $\De_{0}$ is arbitrary and we could choose
 $\De_{0}$ from a refinement of the partition and proceed
 identically).  This induced system is a full branch Markov
 map $F = f^{R} : \cup_{\ell} \De^{(\ell)}_{0} \to \De_{0}$
 where ${\{\De_{0}^{(\ell)}\}}_\ell$ is a countable
 partition (the $\De_{0}^{(\ell)}$ are open sets) of a full
 measure subset of $\De_{0}$ and the first return time
 function $R:\Delta_0\to\bN$ is constant on each
 $\De^{(\ell)}_{0}$.  We define the induced return map
 $\widehat F$ over $\widehat\Delta:=\pi^{-1}\Delta_0$ by
 $\widehat F(x)=P^{R(\pi(x))}(x)$ and the induced return
 time $r:=\sum_{j=0}^{R\circ\pi-1} \tau
   \circ P^j$.  Let $\meb$ denote the normalised
 restriction of the Riemannian volume to $\Delta_0$.  For
 future convenience let $\widehat{\Delta}_{0}^{(\ell)} :=
 \pi^{-1} {\Delta}_{0}^{(\ell)}$ for each $\ell$.  Observe
 that $\widehat F$ and $r$ are the return map and return
 time respectively of the flow $X^{t}$ to the section
 $\widehat \Delta \subset U$ since $F$ was chosen as the
 first return of $f$ to $\Delta_{0}$.  Let $\Sec_{r} = \{
 (x,u) : x \in \widehat \Delta, 0\leq u < r( x)\}$ be the
 phase space of the suspension semiflow\footnote{Note that
   the skew product $\widehat{F}$ is invertible on its image
   but it is not onto. It is in this sense that the
   corresponding suspension is a semiflow and not a flow and
   corresponds to the fact that the attractor we are
   considering has zero volume.  This is the same use of
   terminology as~\cite{AvGoYoc}.} $F_{t} : \Sec_{r} \to
 \Sec_{r}$ which is defined by
 \begin{equation}
  \label{eq:defFt}
  F_{t}(x,u) = 
  \begin{cases}
   (x, u+t) &\text{whenever $u+t < r( x)$}\\
   (\widehat F(x), 0) &\text{whenever $0 \leq u+t - r( x) < r(\widehat{F}(x))$}.
  \end{cases}
 \end{equation}
 The flow is defined for all $t\geq 0$ assuming that the
 semigroup property ${F}_{t+s} = {F}_{t} \circ {F}_{s}$ holds. 
 The suspension flow is
conjugated to the original flow $\flow{t} : U \to U$ for
 some neighbourhood $U$ of the attractor by 
 \begin{equation}
 \label{eq:conjugacy}
 \Phi:\Sec_r\to
 U;
 \quad
  (x,u)\mapsto X^u(x).
 \end{equation}
 Note that $\Phi$ is invertible since $\widehat{F}$ is the first return to $\widehat{\Delta}_{0}$.
Furthermore $\Phi$ directly inherits the good regularity of $\flow{t}$.   
   
   In summary, up until this point, this section has been devoted to choosing the local section
   $\widehat \Delta \subset U$ and so representing $\flow{t}$ as a suspension with return map 
   $\widehat F$ and return time $r$.
   The special feature of this choice is that $r$ is constant on the 
 stable leaves in this local section, it preserves the required regularity 
   and is of the correct form to apply the results of~\cite{AvGoYoc} as we will see shortly.
The previous choice of local section $\Sec$ produced a transitive Markov return map $P$ and bounded return time $\tau$. Unfortunately the quotiented return map $f$ need not be full branch.
On the other hand $F$ is full branch Markov (countable partition) but now the return time could be unbounded.

 \begin{lemma}
 \label{lem:markovmap}
Let $F: \cup_\ell \De^{(\ell)}_0 \to \De_0$ be as defined above. 
The following hold:
 \begin{enumerate}
\item 
For each $\ell$,
  $F :\Delta_{0}^{(\ell)} \to\Delta_{0}$ is a $\cC^{1}$ diffeomorphism;
\item
 There exists $\la\in (0,1)$ such that 
  $\norm{\smash{DF^{-1}(x)}}\leq \lambda$ for all $x\in \cup_\ell \De^{(\ell)}_0$;
\item Let $J$ denote the inverse of the Jacobian of $F$ with respect to $\meb$.
 The function $\log J$ is  $\cC^1$ and
  $\sup_h \norm{ D((\log J)\circ h)}_{\cC^{0}}<\infty$, where the supremum is
  taken over every inverse branch $h$ of $F$.
\end{enumerate}
\end{lemma}

The statement of the above lemma is precisely the definition
of a $\cC^2$ uniformly-expanding full-branch Markov
map~\cite[Definition~2.2]{AvGoYoc}.  The only difference is
that the reference is more general, the domain is there
required just to be what they term a ``John
domain''~\cite[Definition 2.1]{AvGoYoc}, it is immediate
that $(\Delta_{0},\meb)$ satisfies the requirements since
$\Delta_{0}$ is a $\cC^{2}$ disk and $\meb$ is the
restriction of a smooth measure to $\Delta_{0}$.

\begin{proof}[Proof of Lemma~\ref{lem:markovmap}]
The required
regularity of $F$ and $\log J$ are satisfied since $F$ is
$\cC^{2}$.
The uniform hyperbolicity of the flow means that there exists 
$C>0$, $\tilde\lambda\in (0,1)$ such that 
$\norm{\smash{Df^{-n}(x)}}\leq C \tilde\lambda^{n}$ for all $x\in \De_{0}$, $n\in \bN$.
 The definition of the induced return map was
   based on the choice of some element of the Markov
   partition which we denoted by $\De_{0}$.  By choosing
   first a refinement of the Markov partition and then
   choosing $\De_{0}$ from the refined partition we may
   guarantee that the inducing time $R$ is as large as we
   require.
  Since  $\norm{\smash{DF^{-1}(x)}}\leq C \tilde\lambda^{R(x)}$   the uniform expansion estimate of item (2) follows for $\lambda = C \sup_{x} \tilde\lambda^{ R(x)} \in(0,1)$.
Finally, since
$\norm{D(\log Jf)}_{\cC^{0}}$ is bounded
\begin{equation}
 \label{eq:Droh}
 \begin{aligned}
  \norm{ D((\log JF)\circ h) }_{\cC^{0}}
  &=
  \norm{D(\smash{\sum_{j=0}^{R-1}}\log Jf\circ f^j\circ h) }_{\cC^{0}}\\
  &\le
  \sum_{j=0}^{R-1}\norm{D(\log Jf)(f^j\circ h)  D(f^j\circ
  h) }_{\cC^{0}}
  \\
  &\le
  C\sum_{j=0}^{R-1}\norm{D(f^{j}\circ h)}_{\cC^{0}}
  \le C\sum_{j\ge 0}^{R-1} \lambda^{R-j}
  \leq C\sum_{k\ge 0}^{\infty} \lambda^{k}.
\end{aligned}
\end{equation}
is uniformly bounded for all inverse branches $h$ of $F$,
proving item (3).  
\end{proof}

 Since $F$ is a uniformly expanding full-branch Markov map there exists
  a unique invariant probability measure which is absolutely continuous with
  respect to $\meb$. We denote this by $\nu$  and its density by $\vfi = \frac{d\nu}{d\meb}$. 
  Moreover we know that $\vfi\in \cC^1(\Delta_0,\RR^+)$ and is bounded
  away from zero.

 Let $d_{u}$, $d_{s}$ denote the dimension of $ \bE^{u}$, $ \bE^{s}$ respectively.
Taking advantage of the smoothness of the stable foliation and the smoothness of the section we may assume that $\widehat \Delta = \Delta_{0} \times \Omega$
where $\Delta_{0}$ is a $d_{u}$-dimensional ball, $\Omega$ is a $d_{s}$-dimensional ball (in particular compact) and that $\widehat{F} : (x,z) \mapsto (Fx, G(x,z))$ where $F$ is as before the uniformly expanding $\cC^{2}$ Markov map and $G$ is $\cC^{2}$ and contracting in the second coordinate. This puts us in the setting of~\cite{BM2015}.

Given $v:\widehat{\Delta}_{0}\to\bR$, define $v_+,v_-:\Delta_{0}\to\bR$ by setting
$v_+(x)=\sup_{z}v(x,z)$,
$v_-(x)=\inf_{z}v(x,z)$.
Using that $\nu$ is an $F$-invariant probability measure and the contraction in the stable direction we know
that the limits
\begin{equation}
\label{eq:defeta}
\lim_{n\to\infty}\int_{\Delta_{0}} {(v\circ \widehat F^n)}_+\,d\nu\quad\text{and}\quad
\lim_{n\to\infty}\int_{\Delta_{0}} {(v\circ \widehat F^n)}_-\,d\nu
\end{equation}
exist and coincide for all $v$ continuous.  Denote the common limit by $\eta(v)$.
This defines an $\widehat{F}$-invariant probability measure $\eta$ on $\widehat{\Delta_{0}}$ and $\pi_{*}\eta=\nu$ (see, for example,~\cite[Proposition 1]{BM2015}).

Let $\cL$ denote the transfer operator given by $\int_\Delta  g\circ F \cdot v \,d\nu=\int_\Delta g \cdot \cL v\,d\nu$.
By \cite[Propostion 2]{BM2015}, for any $v\in\cC^{0}(\widehat{\Delta},\bR)$,
 the limit
 \begin{equation}
 \label{eq:defetax}
	 \eta_x(v):=\lim_{n\to\infty} (\cL^{n}v_n)(x),\quad v_n(x):=v\circ \widehat F^n(x,0),
	 \end{equation}
	 exists for all $x\in\Delta$ and defines a probability measure supported on $\pi^{-1}(x)$ (without loss of generality we may assume that $0$ denotes some element of $\Omega$).
	 Moreover $x\mapsto \eta_{x}(v)$ is continuous and 
\begin{equation}
\label{eq:connecteta}
\eta(v) = \int_{\Delta} \eta_{x}(v) \ d\nu(x).
\end{equation}
Despite the fact that $\eta$ is singular along stable manifolds, we can take advantage of the regularity of the skew product form of $\widehat{F}$ (due to the regularity of the stable foliation) and the uniform hyperbolicity in order to prove rather good regularity for the  decomposition of $\eta$ into ${\{\eta_{x}\}}_{x\in \Delta_{0}}$.
By  \cite[Proposition 9]{BM2015}, this decomposition is $\cC^{1}$ in the sense that
 there exists $C>0$ such that, for any open set $\omega \subset \Delta$ and
for any $v\in \cC^1(\widehat \Delta,\bR)$, the function $x\mapsto \bar{v}(x) :=  \eta_x(v)$ is $\cC^{1}$ and
\begin{equation}
\label{eq:C1estimate}
	\sup_{x\in \omega}\norm{D\bar{v}(x)} \leq C \sup_{(x,z)\in \pi^{-1}\omega}\abs{v(x,z)} 
	+
	C \sup_{(x,z)\in \pi^{-1}\omega}\norm{Dv(x,z)}. 
\end{equation}

Since $\eta$ is an $\widehat F$-invariant probability measure $\widehat{\eta} := \frac{1}{\eta(r)} \eta \times \Leb$
is an ${F}_{t}$-invariant probability measure. 
Since this measure has absolutely continuous  conditional measure on unstable manifolds due to the connection to the  absolutely continuous $\nu$, 
we know that $\mu = \Phi_{*}\widehat{\eta}$  is the unique SRB measure for $\flow{t}:U \to U$.
\begin{remark}
In the following we will crucially use the result of \cite{AvGoYoc} concerning exponentially mixing hyperbolic suspension semiflows and consequently it is essential that $\nu$ is absolutely continuous with respect to Riemannian volume on $\Delta_{0}$. This is the reason we obtain a result for the SRB measure of the Axiom~A attractor and not for any other Gibbs measure.   
\end{remark}

Recall that the uniform hyperbolicity of the original flow implies
   that there exists $\kappa \in (0,1)$ and $C>0$ such that,
   for all $w_1,w_2 \in \Sec_i$ in the same local stable
   leaf, i.e. $\pi(w_1)=\pi(w_2)$, we have $ d(\widehat{F}^{n}
   w_1, \widehat{F}^{n} w_2) \leq C \kappa^{n} d(w_1,w_2)$ for all $n\in \bN$.
   As previously, in the proof of Lemma~\ref{lem:markovmap}, by choosing
   first a refinement of the Markov partition and then
   choosing $\De_{0}$ from the refined partition we may
   guarantee that the inducing time $R$ is a large as we
   require so that there exists contraction at each iteration by $\hat F$
   or, in other words, we may take $C=1$. For that choice of $\Delta_0$ we have
   the following result.

\begin{lemma}
\label{lem:hypskewproduct}
 Let $\widehat{F}: \cup_\ell \widehat{\Delta}_{0}^{(\ell)} \to \widehat{\Delta}_{0} $,  $F: \cup_\ell \Delta_{0}^{(\ell)} \to \Delta_{0} $ and $F$-invariant probability measure $\nu$  be as defined above. 

\begin{enumerate}
\item 
There exists a continuous map $\pi : \widehat \Delta_{0}
  \to \Delta_{0}$ such that $F\circ \pi = \pi \circ \widehat{F}$ whenever both members of the equality are defined;
\item
 There exists an $\widehat{F}$-invariant probability measure
  $\eta$ giving full mass to the domain of definition of $\widehat{F}$;
\item 
There exists a family of probability measures
  ${\{\eta_x\}}_{x\in \Delta_{0}}$ on $\widehat{\Delta}_{0}$ which is a
  disintegration of $\eta$ over $\nu$, that is, $x\mapsto
  \eta_{x}$ is measurable, $\eta_{x}$ is supported on
  $\pi^{-1}(x)$ and, for each measurable subset $A$ of
  $\widehat\De_{0}$, we have $\eta(A)=\int \eta_x(A)\,d\nu(x)$.
  Moreover, this disintegration is smooth: we can find a
  constant $C>0$ such that, for any open subset $\omega\subset
  \cup_{\ell} \Delta_{0}^{(\ell)}$ and for each $u\in \cC^1(\pi^{-1}(\omega))$,
  the function $\tilde u : \omega \to \RR, x\mapsto\tilde u(x):=\int
  u(y)\,d\eta_x(y)$ belongs to $\cC^1(\omega)$ and satisfies
  \begin{align*}
  \sup_{x\in \omega} \|D\tilde u(x)\|
  \leq C \sup_{(x,z)\in \pi^{-1}\omega}\abs{v(x,z)} 
	+  C \sup_{y\in \pi^{-1}(\omega)} \|Du(y)\|.
\end{align*}
\item 
There exists $\kappa \in (0,1)$ such that, for all
  $w_1,w_2 $ such that $\pi(w_1)=\pi(w_2)$, we have $ d(\widehat F w_1,
  \widehat F w_2) \leq
  \kappa d(w_1,w_2).  $
\end{enumerate}
\end{lemma}

The statement of the above proposition corresponds precisely with {\cite[Definition~2.5]{AvGoYoc}} and says, in their terminology, that $\widehat{F}$ is a \emph{hyperbolic skew-product} over $F$.

 \begin{proof}[Proof of Lemma~\ref{lem:hypskewproduct}]
 Item (1) is clear.
 Item (2) follows from the definition by the limits~\eqref{eq:defeta}.
 Item (3) follows from the definition~\eqref{eq:defetax} and the estimates~\eqref{eq:C1estimate}.
Property (4) is a consequence of the choice of $\Delta_0$.
\end{proof}

\begin{lemma}
\label{lem:goodroof}
Let $F: \cup_\ell \De_{0}^{(\ell)} \to \De_{0} $ and
 $r: \cup_\ell \De_{0}^{(\ell)} \to \bR_{+} $ 
 be the full branch Markov map with countable partition
 and the induced return time as defined before.
\begin{enumerate}
\item
There exists $r_0>0$ such that $r$ is bounded from below by $r_0$;
\item
There exists  $K>0$ such that $|D(r\circ h)| \leq K$ for
 every inverse branch $h$ of $F$.
\end{enumerate}
 \end{lemma}
\begin{proof}
We  note that $R(x)\underline\tau\le r(x)
   \le R(x) \bar\tau$ and so both items
    are consequences of the
   boundedness of $\tau$ (for item (2) we just
   follow the same estimates~\eqref{eq:Droh} in the
   proof of Lemma~\ref{lem:markovmap}).
\end{proof}

\begin{lemma}
\label{lem:exptails}
Let $r: \cup_{\ell} \De^{(\ell)}_{0} \to \bR_{+}$ be defined as before.
 There exists $\sigma_{0} > 0$ such that 
 $\int e^{\sigma_{0} r} \, d\meb <\infty$.
\end{lemma}
\begin{proof}
Since $R$ was defined as the first return time of the uniformly expanding map $f$ to $\Delta_{0}$ we know\footnote{Recall that $f:\cup_{k}\Delta_{k}\to \Delta$ is a Markov expanding map with a finite partition (writing ${\{\Delta_{k}\}}_{k}$ instead of ${\{\Delta_{i,j}\}}_{i,j}$ for conciseness)
and that $F$ was chosen as the first return map of $f$ to the partition element $\Delta_{0}$ (return time denoted by $R$).
By transitivity and the finiteness of the original partition there exists $\beta>0$, $n_{0}\in \bN$,
$Q_{k}\subset \Delta_{k}$ such that $f^{n_{0}}(x) \in \Delta_{0}$ for all $x\in Q_{k}$ and
$\meb(Q_{k}) \geq \beta\meb(\Delta_{k})$.
Let $A_{n}:=\{x : R(x) \geq n\}$ and note that $A_{n}$ is the disjoint union of elements of the 
$n$\textsuperscript{th}-level refinement of the partition. Using bounded distortion ($D>0$)
\[
\frac{\meb(A_{n}\cap f^{-n}Q_{k})}{\meb(A_{n}\cap f^{-n}\Delta_{k})}
\geq
D^{-1} 
\frac{\meb(Q_{k})}{\meb(\Delta_{k})}.
\]
This means that $\meb(A_{n}\cap f^{-n}Q_{k}) \geq D^{-1}\beta \meb(A_{n})$, i.e., a fixed proportion
of the points which have not yet returned  to $\Delta_{0}$ will return to $\Delta_{0}$ within $n_{0}$
iterates.}
 that there exists $\alpha>0$, $C>0$ such that
 $\meb(\{ x\in \Delta_{0} : R(x) \geq  n \})\leq Ce^{-\alpha n}$ for all $n \in \bN$. 
 As $\tau$ is uniformly bounded and $r= \sum_{j=0}^{R-1} \tau \circ f^j$ the estimate of the lemma follows.
\end{proof}

Consider the following cohomology property known as the \emph{uniform non integrability} characteristic of the flow~\cite{Do98}: 
 \begin{itemize}
\item[(\uni):]
There does \emph{not} exist any
 $\cC^{1}$ function
$\gamma : \Delta_0  \to \bR$
 such that $r - \gamma \circ F + \gamma$ is
 constant on each $\smash{\Delta_{0}^{(\ell)}}$.
 \end{itemize}
 The above property is also described as ``$r$ not being
 cohomologous to a locally constant
 function''. 

\begin{lemma}
Suppose that
 $\widehat{F}: \cup_\ell \widehat{\De}_{0}^{(\ell)} \to \widehat{\De}_{0} $, 
 $r: \cup_{\ell} \De^{(\ell)}_{0} \to \bR_{+}$, 
and ${F}_{t}$ the suspension semiflow on $ \Sec_{r}$ preserving the measure $\widehat{\eta}$
are defined as before.
Further suppose that assumption (\uni) holds.
Then there exist $C>0$, $\delta>0$ such that, for all $\phi,\psi \in \cC^{1}(\Sec_{r})$, and for all $t\geq 0$,
\[
\abs{\int \phi\cdot \psi \circ {F}_{t} \ d\widehat{\eta}  - \int \phi  \ d\widehat{\eta} \cdot \int  \psi  \ d\widehat{\eta}} \leq 
C \norm{\phi}_{\cC^{1}}\norm{\psi}_{\cC^{1}}e^{-\delta t}.
\]
\end{lemma}
\begin{proof}
  In order to prove the lemma we make the connection to the
  systems studied in~\cite{AvGoYoc}.
  Lemma~\ref{lem:markovmap} corresponds
  to~\cite[Definition~2.2]{AvGoYoc};
  Lemma~\ref{lem:goodroof} combined with assumption (\uni)
  corresponds to~\cite[Definition~2.3]{AvGoYoc};
  Lemma~\ref{lem:exptails} corresponds
  to~\cite[Definition~2.4]{AvGoYoc}; and finally
  Lemma~\ref{lem:hypskewproduct} corresponds
  to~\cite[Definition~2.5]{AvGoYoc}.  This implies that the
  assumptions of~\cite[Theorem~2.7]{AvGoYoc} are satisfied
  and consequently that the suspension semiflow ${F}_{t}$
  mixes exponentially for $\cC^1$ observables.
\end{proof}
This ensures that the
   original flow $\flow{t} : U \to U$ also mixes
   exponentially for $\cC^1$ observables: if $\phi,\psi:U\to\RR$
   are $\cC^1$, then $\phi\circ\Phi,\psi\circ\Phi$ are
   $\cC^1$ observables on $\Sec_{r}$ ($\Phi$ is the conjugacy~\eqref{eq:conjugacy}) and so
   \begin{align*}
     \int \phi\cdot (\psi\circ X^t) \,d\mu=\int (\phi\circ\Phi)
     \,\cdot\, (\psi\circ X^t\circ\Phi) \,d\widehat{\eta} =\int
     (\phi\circ\Phi) \,\cdot\, (\psi\circ\Phi) \circ {F}_t\,d\widehat{\eta}
\end{align*}
goes exponentially fast to zero.  By an
approximation argument~\cite[Proof of Corollary 1]{Do98} this implies exponential mixing for
H\"older observables.

A more direct proof of Theorem~\ref{thm:main} would be to establish the open and denseness of the (\uni) assumption 
with respect to a given choice of cross-sections. 
However, since Theorem~\ref{thm:expmix} in terms of non joint-integrability of stable and unstable foliations is of independent interest, we choose now to make the connection between (\uni) and non joint-integrability.
 
  \begin{lemma}
  \label{lem:UNI}
  Suppose that
 $\widehat{F}: \cup_\ell \widehat{\De}_{0}^{(\ell)} \to \widehat{\De}_{0} $
 and
 $r: \cup_{\ell} \De^{(\ell)}_{0} \to \bR$
are defined as above.
 Property  (\uni) fails if and only if  the stable and unstable foliations  
of the underlying Axiom~A flow $\flow{t}:U\to U$  are jointly integrable.
  \end{lemma}

 \begin{proof}
  First suppose that (\uni) fails and so there exists a $\cC^{1}$ function $\gamma : \cup_{\ell} \Delta^{(\ell)}_{0}
  \to \bR$ such that $r - \gamma \circ F + \gamma$ is  constant on each $ \Delta^{(\ell)}_{0}$.
  For notational convenience extend $\gamma$ to $\hat\Delta^{(\ell)}_{0}$ as the function which is constant along local stable manifolds.
  This means that $r - \gamma \circ \hat{F} + \gamma$ is  constant on each $ {\hat\Delta}^{(\ell)}_{0}$.
 %
  For each $\ell$ fix some $x_{\ell} \in
  \overline{\hat\Delta_{0}^{(\ell)}}$ such that $\gamma(x) \leq
  \gamma(x_{\ell})$ for all $x\in {\hat\Delta}_{0}^{(\ell)}$
  (recall that $\gamma:\Delta_0\to\bR$ is a
    $C^1$ function).  Choosing, if required, a refinement
  of the Markov partition we guarantee that
  $\gamma(x_{\ell}) - \gamma(x) \leq \inf r$.  Since
  $\gamma$ is a $\cC^1$ function consider (for each $\ell$)
  the $(d_{u}+d_{s})$-dimensional hypersurface
 \begin{equation}
 \label{eq:surfacesss}
\hypersur{\ell}
:= 
\left\{(x, \gamma(x_{\ell}) -\gamma(x)) : x \in {\hat\Delta}_{0}^{(\ell)}\right\}
\subset  \Sec_{r}.
 \end{equation}
 We claim that the return time function to this family of local cross-sections $\cup_{\ell} \hypersur{\ell}$ is locally constant. 
 For some $\ell, \ell'$ consider the set $(x,a) \in \hypersur{\ell}$ such that $\hat{F}(x) \in  {\hat\Delta}_{0}^{(\ell')}$
 For $x$ such that $0 < t - r(x)  < r (Fx)$ 
\[
 F_t (( x, \gamma(x_{\ell}) -\gamma(x) ))
 	 = ( \hat{F}x, t + \gamma(x_{\ell}) -\gamma(x) - r(x))
\]
 Then the first $t>0$ so that
 $
 F_t ((x, \gamma(x_{\ell}) -\gamma(x)))  \in \cup_{\ell'}\hypersur{\ell'}
 $
  is given as a solution of 
  $
  t+ \gamma(x_{\ell})-\gamma(x) - r(x) = \gamma(Fx_{\ell'}) -\gamma (Fx)
  $
  or, equivalently, 
  $$
    t = [\gamma(Fx_{\ell'})  - \gamma(x_{\ell})] + [r(x) -\gamma (Fx) + \gamma(x)],
  $$
which is locally constant since it does not depend on $x \in {\hat\Delta}^{(\ell)}_{0} \cap \hat{F}^{-1}{\hat\Delta}^{(\ell')}_{0}$.
 This proves that 
 \begin{equation}\label{eq:surface}
\left\{(x, a+ \gamma(x_{\ell}) -\gamma(x) : x \in {\hat\Delta}^{(\ell)}_{0}, a\in \mathbb R^+ \right\}
 \end{equation}
 defines a codimension-one invariant $\cC^{1}$-foliation
 transversal to the flow direction, which implies that the
 stable and unstable foliations are jointly integrable.
 This shows that if (UNI) fails, then we
   obtain joint integrability.  The other direction of the
 statement is well known~\cite{FMT}.
 \end{proof}
 
%
 
 \section{Robust exponential mixing}
 \label{sec:mainproof}
 
 The purpose of this section is to prove
 Theorem~\ref{thm:main}.  We make use of 
 Theorem~\ref{thm:expmix} by constructing open sets of Axiom A flows that have
 attractors whose stable and unstable foliations are not jointly integrable in a robust way. 
Let $\M$
 be a Riemannian manifold of dimension $d\geq 3$.  It is
 sufficient to prove that there exists an
   open set of vector fields supported on the
   $d$-dimensional hypercube $D^{d} := (0,1)^{d}$ satisfying
   the conclusion of Theorem~\ref{thm:main}.

 \begin{lemma}\label{lem:exists}
   For any $d \geq 3$ there exists a vector field $X \in
   \fX^{3}(\bR^{d})$ such that $\flow{t} : \bR^{d} \to
   \bR^{d}$ exhibits an Axiom~A attractor contained within
   $(0,1)^{d}$. Moreover the domination
   condition~\eqref{eq:kdomination} holds.
 \end{lemma}
 \begin{proof}
 The \emph{Plykin attractor}~\cite[\S8.9]{robinson1999} is a smooth diffeomorphism of a bounded subset of $\bR^{2}$ which exhibits an Axiom~A attractor. 
By \cite[\S2]{newhouse1978} we may use the existence of an Axiom~A diffeomorphism on $D^{2}$ to construct an Axiom~A flow exhibiting an attractor in a bounded subset of $D^{2} \times S^{1}$. This can be embedded in $\bR^{3}$ as a solid torus. It is a simple matter to ensure that the contraction in the stable direction is much stronger that the expansion in the unstable direction so that the domination condition~\eqref{eq:kdomination} holds. 
An attractor of higher dimensions is achieved by a similar construction but with additional uniformly contracting directions added.
 \end{proof}
 
 The above construction is far from being the only possibility. To construct a four dimensional flow another obvious choice is to start with the \emph{Smale Solenoid Attractor}. This is an Axiom~A diffeomorphism in 
$\bR^{n}$ which exhibits an attractor $\Lambda$ with one-dimensional unstable foliation and
 $(n-1)$-dimensional stable foliation. Constructing a flow from this as per the above proof gives an $(n+1)$-dimensional Axiom~A flow exhibiting an attractor. For a wealth of possibilities in higher dimension we may take advantage of the work of Williams~\cite{Williams1974} on expanding attractors. By this the expanding part of the the system is determined by a symbolic system of an $n$-solenoid, this means that the expanding part of the system may be any dimension desired. Then he shows that the stable directions may be added and the whole system embedded as a vector field on $\bR^{d}$.
 
Since the domination condition~\eqref{eq:kdomination} is
open  we obtain the following as an immediate
consequence of Lemma~\ref{lem:exists} and of the regularity
given by Theorem~\ref{thm:regularity}: There exists a
$C^1$-open subset $\mathcal U\subset \mathfrak{X}^2(M)$ of
Axiom~A flows exhibiting an attractor $\Lambda$
with $\cC^{2}$ stable foliation.  Then, using
Theorem~\ref{thm:expmix}, for any $X\in \cU$ the
corresponding  attractor for the Axiom A flow
$\flow{t} : \M \to \M$ admits a $C^2$ stable foliation and,
consequently the unique SRB measure mixes exponentially as long as  the stable foliation and unstable foliation are not jointly integrable. 
Since it is known since Brin (see~\cite{FMT} and references within) that the non integrability of the stable and unstable foliations is $\cC^r$-open and dense 
this completes the proof of Theorem~\ref{thm:main}.

\appendix

\section{Erratum to ``Open sets of Axiom A flows with
  Exponentially Mixing Attractors''}
\label{sec:erratum-open-sets}

In ``Open sets of Axiom A flows with exponentially mixing
attractors'' there was an oversight concerning the
regularity of the Markov partition for higher dimensional
flows.  Previously~\cite[p.2978]{ArBuVa}\footnote{In this
  file, the comment following the statement of
  Lemma~\ref{lem:markovmap}.} we claimed that any element of
a Markov partition of a non-trivial attractor for an Axiom A
vector field, after quotienting out the stable leaves, is a
{$C^2$ disk, hence a} ``John domain'' \cite[Definition
2.1]{AvGoYoc}.  The argument for exponential mixing
presented in the paper relies on the application of
results~\cite[Theorem 2.7]{AvGoYoc} where the ``John
domain'' condition is required.

Bowen~\cite{Bowen1978} showed that, for higher
dimensional systems, the boundaries of the Markov partition
elements cannot be smooth (here smooth means piecewise~\(\cC^{1}\)), in particular the objects of interest cannot be expected to be $C^2$ disks as previously claimed. In general there is no reason to expect that the elements of the Markov partition are John domains.
 When the unstable bundle is higher dimensional and the expansion is not isotropic then there seems no hope that the sets are John domains (for evidence of this consult the estimates and comments in \cite[\S A.2]{BW17}).

Here we show that the originally claimed result remains valid.
First observe that, as previously described~\cite[\S4]{ArBuVa}, for any \(d\geq 3\), it is possible to construct 
 examples of vector fields with Axiom~A attractors which have 1D unstable bundle and which satisfy the requirements of the rest of the construction.
If the unstable bundle is 1D then the relevant element of the Markov partition is a John domain. This is because it is constructed \cite[\S3]{ArBuVa} as the connected subset of a local unstable manifold. Such a connected 1D set is automatically a John domain.
The above observation gives immediately the main results as follows.
\begin{maintheorem}\label{thm:main}
  Given any  Riemannian manifold $\M$ of dimension
  $d\geq 3$ there exists a $\cC^{1}$-open subset of
  $\cC^3$-vector fields $\cU \subset \fX^3(M)$ such that for
  each $X\in\cU$ the associated flow is Axiom~A and exhibits
  a non-trivial attractor  {\(\Lambda\)} which mixes exponentially with
  respect to the unique SRB measure  {of \(\Lambda\)}.
\end{maintheorem}
 \begin{maintheorem}\label{thm:expmix}
   Suppose that $\flow{t} : \M \to \M$ is a $\cC^2$ Axiom~A
   flow, $\Lambda$ is an attractor  {(in particular closed and topologically transitive)},  that the stable foliation is
   $\cC^2$ and that the unstable bundle is one-dimensional.   
   If the stable and unstable foliations
    are not jointly integrable
    then the flow mixes exponentially with
   respect to the unique SRB  {of $\Lambda$}.
    \end{maintheorem}
 \noindent
The first theorem stands exactly as previously stated~\cite[Theorem A]{ArBuVa}.
The seconds remains as previously stated~\cite[Theorem B]{ArBuVa} except for the addition of the assumption that the unstable bundle is one-dimensional.

The restriction that each of the flows constructed must have unstable bundle which is 1D could be removed using further improvements of the methods. 
Studying exponentially mixing Anosov flows Butterley \& War~\cite{BW17} showed that elements of the Markov partition of the system obtained by quotienting along local stable manifolds can be identified with a finite number of connected subsets of $\mathbb R^d$ which satisfy some weak geometric properties~\cite[Appendix A]{BW17}. Moreover these properties suffice for showing exponential mixing of the relevant suspension semiflow (a condition inspired by John domains but weaker is used). 
It appears that the resultant theorem~\cite[Theorem 1]{BW17} holds true (with the same proof) in the case of Axiom~A attractors and not only for Anosov flows (i.e., Theorem~\ref{thm:expmix} without requiring the 1D unstable bundle).
 Since the argument involves reducing the problem to an expanding semiflow by quotienting the key is to show that the quotient system satisfies the assumptions of the theorem~\cite[Theorem 3]{BW17} concerning exponential mixing of expanding semiflows. All the hyperbolicity properties are as before, the new details are the geometric properties (as proven in \cite[Appendix A]{BW17}). The key observation is that the systems of interest are attractors and not just Axiom~A which means the system after quotienting along local stable manifolds sufficiently resembles the one obtained if the original system were Anosov. This is because, for a hyperbolic attractor, the local unstable manifold of each point of the attractor is contained within the attractor.



\begin{thebibliography}{10}

\bibitem{ArBuVa}
V.~Ara\'ujo, O.~Butterley, and P.~Varandas.
\newblock Open sets of {A}xiom {A} flows with exponentially mixing attractors.
\newblock {\em Proc. Amer. Math. Soc.}, 144:2971--2984, 2016.

  
\bibitem{AM2015}
V.~Ara{\'u}jo and I.~Melbourne.
\newblock Exponential decay of correlations for nonuniformly hyperbolic flows with a $C^{1+\alpha}$ stable foliation.
\newblock Preprint ArXiv:1504.04316.  


\bibitem{ArVar}
V.~Ara{\'u}jo and P.~Varandas.
\newblock Robust exponential decay of correlations for singular-flows.
\newblock {\em  Comm. Math. Phys.}, 311:215--246, 2012.

\bibitem{AvGoYoc}
A.~Avila, S.~Gou{\"e}zel, and J.-C. Yoccoz.
\newblock {Exponential mixing for the Teichm{\"u}ller flow}.
\newblock {\em {Publ. Math. Inst. Hautes {\'E}tudes Sci.}}, {104}:{143--211},
  {2006}.


\bibitem{BaVal2005}
V.~Baladi and B.~Vall{\'e}e.
\newblock {Exponential decay of correlations for surface semi-flows without
  finite Markov partitions}.
\newblock {\em {Proc. Amer. Math. Soc.}}, {133}({3}):{865--874}, {2005}.

\bibitem{Bo73}
R.~Bowen.
\newblock {Symbolic dynamics for hyperbolic flows}.
\newblock {\em {Amer. J. Math.}}, {95}:{429--460}, {1973}.

\bibitem{Bo75}
R.~Bowen.
\newblock {\em {Equilibrium states and the ergodic theory of Anosov
  diffeomorphisms}}, volume {470} of {\em {Lect. Notes in Math.}}
\newblock {Springer Verlag}, {1975}.

\bibitem{Bowen1978}
R.~Bowen.
\newblock Markov partitions are not smooth.
\newblock {\em Proceedings of the American Mathematical Society},
  71(1):130--132, 1978.



  
\bibitem{BM2015}
O.~{Butterley} and I.~{Melbourne}.
\newblock {Disintegration of invariant measures for
  hyperbolic skew products}.
\newblock {\em {Israel J. of Math.}}, {219}({1}):{171--188}, {2017}.

\bibitem{BW17}
O.~Butterley \& K.~War.
\newblock Open sets of exponentially mixing Anosov flows.
\newblock To appear in {\em J. Eur. Math. Soc.}, 2018. 


\bibitem{chernov98}
N.~I. Chernov.
\newblock {Markov approximations and decay of correlations for Anosov flows}.
\newblock {\em {Ann. of Math. (2)}}, {147}({2}):{269--324}, {1998}.

\bibitem{Do98}
D.~Dolgopyat.
\newblock {On decay of correlations in Anosov flows}.
\newblock {\em {Ann. of Math. (2)}}, {147}({2}):{357--390}, {1998}.

\bibitem{dolgopyat98}
D.~Dolgopyat.
\newblock {Prevalence of rapid mixing in hyperbolic flows}.
\newblock {\em {Ergod. Th. \& Dynam. Sys.}}, {18}({5}):{1097--1114},
  {1998}.

\bibitem{Dolgopyat2000}
D.~Dolgopyat.
\newblock {Prevalence of rapid mixing. II. Topological prevalence}.
\newblock {\em {Ergod. Th. \& Dynam. Sys.}}, {20}({4}):{1045--1059},
  {2000}.


\bibitem{FMT}
M.~Field, I.~Melbourne, and A.~T{\"o}rok.
\newblock {Stability of mixing and rapid mixing for hyperbolic flows}.
\newblock {\em {Ann. of Math. (2)}}, {166}:{269{--}291}, {2007}.

\bibitem{HW1999}
B.~Hasselblatt and A.~Wilkinson.
\newblock Prevalence of non-{L}ipschitz {A}nosov foliations.
\newblock {\em Ergod. Th. \& Dynam. Sys.}, 19(3):643--656, 1999.

\bibitem{HP70}
M.~Hirsch and C.~Pugh.
\newblock {Proc. Sympos. Pure Math. (Berkeley 1968)}.
\newblock In {\em {Global analysis}}, volume {XIV}, chapter {Stable manifolds
  and hyperbolic sets}, pages {133--163}. {Amer. Math. Soc.}, {1970}.

\bibitem{HPS77}
M.~Hirsch, C.~Pugh, and M.~Shub.
\newblock {\em {Invariant manifolds}}, volume {583} of {\em {Lect. Notes in
  Math.}}
\newblock {Springer Verlag}, {New York}, {1977}.

\bibitem{liverani2004}
C.~Liverani.
\newblock {On contact Anosov flows}.
\newblock {\em {Ann. of Math. (2)}}, {159}({3}):{1275--1312}, {2004}.

\bibitem{newhouse1978}
S.~Newhouse, D.~Ruelle, and F.~Takens.
\newblock Occurrence of strange {A}xiom {A} attractors near quasiperiodic flows
  on {$T^{m}$},{$\,m\geq 3$}.
\newblock {\em  Comm. Math. Phys.}, 64(1):35--40, 1978/79.

\bibitem{PT93}
J.~Palis and F.~Takens.
\newblock {\em {Hyperbolicity and sensitive-chaotic dynamics at homoclinic
  bifurcations}}.
\newblock {Cambridge University Press}, {1993}.

\bibitem{Pol85}
M.~Pollicott.
\newblock {On the rate of mixing of Axiom A flows}.
\newblock {\em {Invent. Math.}}, {81}({3}):{413--426}, {1985}.

\bibitem{robinson1999}
C.~Robinson.
\newblock {\em {Dynamical systems}}.
\newblock {Studies in Advanced Mathematics}. {CRC Press}, {Boca Raton, FL},
  {Second} edition, {1999}.
\newblock {Stability, symbolic dynamics, and chaos}.

\bibitem{PSW}
C.~Pugh, M.~Shub and A. Wilkinson.
\newblock  H\"older foliations
\newblock  Duke Math. J.,  86: 3, (1997), 517--546.

\bibitem{Ru76a}
D.~Ruelle.
\newblock Zeta-functions for expanding maps and {A}nosov flows.
\newblock {\em  Invent. Math.}, 34(3):231--242, 1976.

\bibitem{ruelle1983}
D.~Ruelle.
\newblock {Flots qui ne m{\'e}langent pas exponentiellement}.
\newblock {\em {C. R. Acad. Sci. Paris S{\'e}r. I Math.}},
  {296}({4}):{191--193}, {1983}.

\bibitem{Sm67}
S.~Smale.
\newblock {Differentiable dynamical systems}.
\newblock {\em {Bull. Am. Math. Soc.}}, {73}:{747--817}, {1967}.

\bibitem{Williams1974}
R.~F. Williams.
\newblock Expanding attractors.
\newblock {\em  Inst. Hautes \'Etudes Sci. Publ. Math.}, 43:169--203, 1974.

\bibitem{Young2002}
L.-S. Young.
\newblock What are {SRB} measures, and which dynamical systems have them?
\newblock {\em Journal of Statistical Physics}, 108(5-6):733--754, 2002.

\end{thebibliography}
\end{document}